\documentclass{article}
\usepackage[margin=1.3in]{geometry}

\usepackage{amsmath,amssymb,amsthm,mathrsfs,esint,color,times,textcomp,sectsty,verbatim,graphicx,enumerate,authblk,amsxtra}
\usepackage{xcolor}
\usepackage[colorlinks=true]{hyperref}
\hypersetup{urlcolor=blue, citecolor=red, linkcolor=blue}
\usepackage[toc,page]{appendix}
\numberwithin{equation}{section}
\theoremstyle{plain}
\newtheorem{theorem}{Theorem}[section]

\newtheorem{lemma}[theorem]{Lemma}
\newtheorem{conjecture}{Conjecture}
\newtheorem{corollary}[theorem]{Corollary}

\theoremstyle{definition}
\newtheorem{definition}[theorem]{Definition}

\newtheorem{remark}[theorem]{Remark}

\newcommand{\Rmnum}[1]{\expandafter\@slowromancap\romannumeral #1@}

\newcommand{\mr}{\mathbb{R}}
\newcommand{\ms}{\mathbb{S}}
\newcommand{\ud}{\mathrm{d}}
\allowdisplaybreaks

\begin{document}
	\title{\Large \bf Higher order Bol's inequality  and its  applications}
	\author{Mingxiang Li \thanks{M. Li, Department of Mathematics, Nanjing University, China,  limx@smail.nju.edu.cn}, Juncheng Wei\thanks{J. Wei, Department of Mathematics,  Chinese University of Hong Kong,  Shatin, NT, Hong Kong,  wei@math.cuhk.edu.hk. }}
\date{}
\maketitle
\begin{abstract}
In the  conformal class of  Euclidean space, we give some volume comparison theorems with help of   Q-curvature. Meanwhile, for compact four dimensional manifolds with  non-negative scalar curvature, we give a  volume rigidity theorem with respect to Q-curvature.  Finally, we make use of these results to give some sufficient and  necessary conditions for the existence of solutions to some conformally invariant  equations which answers an open problem raised by Hyder-Martinazzi (2021, JDE).
\end{abstract}

\medskip

\section{Introduction}

The well-known Bishop-Gromov volume comparison theorem plays an important role in geometry. Specially,  for the complete n-manifold $(M^n,g)$ if  the Ricci curvature satisfies
\begin{equation}\label{Ric geq n-1}
	Ric_g\geq (n-1)g,
\end{equation} then its  volume is less than or equal to that of standard n-sphere with equality holds if and only if $(M,g)$ is isometric to standard sphere.  However, the condition \eqref{Ric geq n-1} seems  very restrictive. One may ask whether we can obtain the same  volume comparison under some weaker assumptions.  In his Ph.D. thesis \cite{Bray}, Bray  tried to relax the condition \eqref{Ric geq n-1} to
\begin{equation}\label{Ric geq  eplsilon_nn-1}
	Ric_g\geq \epsilon_n(n-1)g,
	\quad R_g\geq n(n-1)
\end{equation}
where the  constant $0<\epsilon_n<1$ and $R_g$ is the scalar curvature. Until now, Bray's conjecture is still unsolved but some progresses are obtained in \cite{Yuan} and \cite{Zhang Y}. Another perspective to consider is utilizing the Q curvature and some results have been established by Lin-Yuan \cite{Lin-Yuan} under some special conditions.   In this paper, we want to establish some volume comparison theorems from this  perspective.

Before stating our results, let us briefly recall the  history of Q-curvature. About forty years ago, Q-curvature is introduced by  Branson \cite{Bra}.  For four dimensional manifold $(M^4,g)$, it is defined by
\begin{equation}\label{Q_g}
	Q_g=\frac{1}{6}(-\Delta_g R_g-3|Ric_g|^2+R_g^2).
\end{equation}
Under the conformal transformation $(M^4, \tilde g)=(M^4,e^{2u}g)$, it satisfies the conformal invariant equations
\begin{equation}\label{P_g}
	P_gu+Q_g=Q_{\tilde g}e^{4u}
\end{equation}
where  $P_g$ is the Paneitz operator defined by
$$P_g=(\Delta_g)^2-div_g\left((\frac{2}{3}R_gg-2Ric_g)d\right).$$ For more information about Q-curvature, interested readers can refer to \cite{CY95},  \cite{DM}, \cite{GJMS},  \cite{Gursky}, \cite{WX JFA}, and the references therein.

In the Euclidean space $\mr^n$ where  $n\geq 2$ is an even integer, the Q-curvature of the conformal metric $g=e^{2u}|dx|^2$ satisfies the following conformally invariant  equation
\begin{equation}\label{equ:-Delta ^n/2u=Ke^nu}
	(-\Delta)^{\frac{n}{2}}u(x)=Qe^{nu(x)}\quad \mathrm{on}\;\;\mr^n.
\end{equation}
  When $n=2$, Q-curvature is  exactly the Gaussian curvature of the conformal metric $e^{2u}|\ud x|^2$ on $\mathbb{R}^2$.   By slightly modifying the Ding's lemma in \cite{CL91}, when $Q\leq 1$, we have the volume of $(\mr^2, e^{2u}|dx|^2)$ is greater than or equal to $4\pi$. In general surfaces, it is known as the  Bol's inequality. More details can be found in \cite{Ba},  \cite{Su}, \cite{GHM}, \cite{GM} and the references therein.   In \cite{LX23}, such  Bol's inequality on $\mr^2$ (See Lemma \ref{lem: modified Ding's lemma}) plays a crucial role  in  ruling out the "slow bubble" on torus.

  For higher order cases, it is natural to ask whether a  higher order Bol's inequality holds?  More precisely,  if we give some barriers on Q-curvature, can we obtain some control of the volume?   This  is very important in studying the asymptotic behavior of conformal metrics with null Q-curvature in \cite{Li One bubble for Q}.   Unfortunately, for $n\geq 4$, Chang and Chen's result \cite{Chang-Chen} tells us that the restriction to just the Q curvature alone is insufficient for volume control. In fact, even for the case $Q=1$, the volume could be arbitrarily small (See Theorem 1 in \cite{Chang-Chen}). More details about the volume for constant Q-curvature  can be found in \cite{WY}, \cite{Mar AIH}, \cite{HY}, \cite{Hyder} and the references therein.
   Hence, we need additional assumptions. Inspired by the work of Chang-Qing-Yang \cite{CQY},  we find that  an additional  restriction on the sign of scalar curvature near infinity will help us to obtain the volume comparison theorem.  Besides, we need assume that $Qe^{nu}\in L^1(\mr^n)$. If so, $(\mr^n, e^{2u}|dx|^2)$  is said to be with finite total Q-curvature. We denote the volume of standard n-sphere by $|\ms^n|$. Now,  our first result can be stated as follows.
  \begin{theorem}\label{thm: Q leq n-1 R geq0 half S^n}
  	Given a conformal metric $g=e^{2u}|dx|^2$ on $\mr^n$ where $n\geq 4$ is an even integer with finite total Q-curvature. Suppose that   the Q-curvature
  	$$Q_g\leq (n-1)!$$
  	as well as scalar curvature $R_g\geq 0$ near infinity. Then there holds
  	\begin{equation}\label{volume S^n/2}
  		\int_{\mr^n}e^{nu}\ud x>\frac{ |\ms^n|}{2}.
  	\end{equation}
  \end{theorem}
 \begin{remark}
 	We believe that such lower bound \eqref{volume S^n/2} is not sharp.  In fact,  under additional radial symmetry of $u(x)$, we give a rigidity volume comparison theorem.
 \end{remark}
  \begin{theorem}\label{thm: Q leq n-1 R geq0}
  	Given a conformal metric $g=e^{2u}|dx|^2$ on $\mr^n$ where $n\geq 4$ is an even integer with finite total Q-curvature. Suppose that $u(x)$ is radially symmetric and  the Q-curvature
  	$$Q_g\leq (n-1)!$$
  	as well as scalar curvature $R_g\geq 0$ near infinity. Then there holds
  	$$\int_{\mr^n}e^{nu}\ud x\geq |\ms^n|$$
  with equality holds if and only if $u(x)=\log\frac{2\lambda}{\lambda^2+|x|^2}$ for some constant $\lambda>0$.
  \end{theorem}

  As a dual result, if we give a lower bound of Q-curvature, we also obtain a volume comparison theorem.
    \begin{theorem}\label{thm: Q geq n-1 R geq0}
  	Given a conformal metric $g=e^{2u}|dx|^2$ on $\mr^n$ where $n\geq 4$ is an even integer with finite total Q-curvature. Suppose that $u(x)$ is radially symmetric and  the Q-curvature
  	$$Q_g\geq (n-1)!$$
  	as well as scalar curvature $R_g\geq 0$ near infinity. Then there holds
  	$$\int_{\mr^n}e^{nu}\ud x\leq |\ms^n|$$
  	with equality holds if and only if $u(x)=\log\frac{2\lambda}{\lambda^2+|x|^2}$ for some constant $\lambda>0$.
  \end{theorem}
  \begin{remark}
  	In fact, for $n=2$, Gui and Li give a proof for this result  in Theorem 1.5 of  \cite{GL}. However, their method doesn't work for higher order cases.
  \end{remark}

  	We tend to believe that the radially symmetric assumption on the metric  can be removed.  In section \ref{section: Bol's inequality for polynomial}, we will give some volume comparison theorems for some special Q-curvature without  radially symmetric assumption which further  corroborates our speculation.
  	
  	 For readers'  convenience, we  leave them as  two conjectures.

\begin{conjecture}\label{conjecture}
	Given a conformal metric $g=e^{2u}|dx|^2$ on $\mr^n$ where $n\geq 4$ is an even integer with finite total Q-curvature. Suppose that  the Q-curvature
	$$Q_g\leq (n-1)!$$
and 	the scalar curvature $R_g\geq 0$ near infinity.  Then there holds
	$$\int_{\mr^n}e^{nu}\ud x\geq |\ms^n|$$
	with equality holds if and only if $u(x)=\log\frac{2\lambda}{\lambda^2+|x-x_0|^2}$ for some constant $\lambda>0$ and $x_0\in\mr^n$.
	
\end{conjecture}
\begin{conjecture}\label{conjecture2}
	Given a conformal metric $g=e^{2u}|dx|^2$ on $\mr^n$ where $n\geq 4$ is an even integer with finite total Q-curvature. Suppose that  the Q-curvature
	$$Q_g\geq (n-1)!$$
	and 	the scalar curvature $R_g\geq 0$ near infinity.  Then there holds
	$$\int_{\mr^n}e^{nu}\ud x\leq |\ms^n|$$
	with equality holds if and only if $u(x)=\log\frac{2\lambda}{\lambda^2+|x-x_0|^2}$ for some constant $\lambda>0$ and $x_0\in\mr^n$.
	
\end{conjecture}
By making use of the Chern-Gauss-Bonnet formula
\begin{equation}\label{CGB formula}
	\int_{\ms^n}Q_g\ud\mu_{g}=(n-1)!|\mathbb{S}^n|,
\end{equation}
 we will find that such volume comparison theorems are trivial in  the conformal class of standard sphere $(\ms^n, g)=(\ms^n, e^{2u}g_0).$  For general compact manifolds, can we obtain some volume comparison theorems with help of Q-curvature? Inspired by the work of Gursky (Theorem B in \cite{Gursky}), for a general compact four dimensional manifolds, we give a  volume rigidity theorem as follows.
\begin{theorem}\label{thm: Q_  geq n-1, R_g geq 0 for compact}
	Given  a compact  four dimensional manifold $(M^4, g)$. If the scalar curvature $R_g\geq 0$ and the Q-curvature
	$Q_g\geq 6,$
	then  there holds
	$$V(M^4,g)\leq |\ms^4|$$
	with equality holds if and only if  $(M^4, g)$ is isometric to  $(\ms^4, \tilde g)$ where the metric $\tilde g$ can be written as  $$\tilde g=\left(\frac{2\lambda}{\lambda^2+|x|^2}\right)^2|dx|^2$$
	for some $\lambda>0$ via a stereographic projection.
\end{theorem}
\begin{remark}
	We hope that similar volume comparison theorem still holds for higher dimensional cases.
\end{remark}

The paper is organized as follows. In Section \ref{section:asyptptic}, we study the asymptotic behavior of the normal solutions to \eqref{normal solution} and  prove Theorem \ref{thm: Q leq n-1 R geq0 half S^n}.
In Section \ref{section: Pohozaev identity}, we give some Pohozaev identities with different restrictions which play an important role in the proofs of our main theorems. In Section \ref{section: PCC condition}, we introduce an {\bf s-cone} condition on Q-curvature and study the range of the integrated Q-curvature under the {\bf s-cone} condition. In Section \ref{section: Bol's inequality for polynomial}, we focus on polynomial Q-curvature to derive Bol's inequality. In Section \ref{section: Bol's inequality for radial solution}, we address radially  symmetric solutions and present the proofs for Theorem \ref{thm: Q leq n-1 R geq0} and  Theorem \ref{thm: Q geq n-1 R geq0}.  In section \ref{section: compact}, we will discuss volume comparison theorem by using Q-curvature on compact manifolds and prove Theorem \ref{thm: Q_  geq n-1, R_g geq 0 for compact}. Finally, in Section \ref{section:application}, we offer an applications of our higher-order Bol's inequality and provide a resolution to an open problem posed by Hyder and Martinazzi in \cite{HM}.

\section{Asymptotic behavior of normal solutions}\label{section:asyptptic}

For the sake of convenience, we use the notation $B_R(p)$ to refer to an Euclidean  ball in $\mr^n$ centered at $p\in\mr^n$ with a radius of $R$.  For a function $\varphi(x)$, the positive part of $\varphi(x)$ is denoted as $\varphi(x)^+$
and the negative part of $\varphi(x)$ is denoted as $\varphi(x)^-$.  Set
$\fint_{E}\varphi(x)\ud x=\frac{1}{|E|}\int_{E}\varphi(x)\ud x$
for any measurable set $E$.  For a constant $C$, $C^+$ denotes $C$ if $C\geq 0$, otherwise, $C^+=0$.
Here and thereafter, we denote by $C$ a constant which may be different from line to line. For $s\in \mr$, $[s]$ denotes the largest integer not greater than $s$.

Supposing that $Qe^{nu}\in L^1(\mr^n)$ for the equation \eqref{equ:-Delta ^n/2u=Ke^nu},
we say that the solution  $u(x)$ is a normal solution to \eqref{equ:-Delta ^n/2u=Ke^nu} if $u(x)$ satisfies the integral equation
\begin{equation}\label{normal solution}
	u(x)=\frac{2}{(n-1)!|\mathbb{S}^n|}\int_{\mr^n}\log\frac{|y|}{|x-y|}Q(y)e^{nu(y)}\ud y+C.
\end{equation}
For brevity,  we denote  the normalized integrated Q-curvature as
$$\alpha:=\frac{2}{(n-1)!|\mathbb{S}^n|}\int_{\mr^n}Qe^{nu}\ud x.$$
More details about normal solutions can be found in Section 2 of \cite{Li 23 Q-curvature}.  Here, we assume that both $u(x)$ and $Q(x)$ are smooth functions on $\mathbb{R}^n$. Although similar results can also be obtained under weaker regularity assumptions,  for the sake of brevity, we will focus on the smooth cases throughout  this paper.

For reader's convenience, we repeat  the following lemmas which have  been established in \cite{Li 23 Q-curvature}.
\begin{lemma} \label{lem: L(f)}
	(Lemma 2.3 in \cite{Li 23 Q-curvature}) Consider the normal solution $u(x)$ to \eqref{normal solution} with even integer $n\geq 2$.  For $|x|\gg1$, there holds
	\begin{equation}\label{Lf=-aplha log x+}
		u(x)=(-\alpha+o(1))\log|x|+\frac{2}{(n-1)!|\mathbb{S}^n|}\int_{B_1(x)}\log\frac{1}{|x-y|}Q(y)e^{nu(y)}\ud y
	\end{equation}
	where  $o(1)\to 0$ as $|x|\to\infty$.
\end{lemma}
\begin{lemma}\label{lem: B_1 L(f)} (Lemma 2.4 in \cite{Li 23 Q-curvature})
	Consider the normal solution $u(x)$ to \eqref{normal solution} with even integer $n\geq 2$.  For $|x|\gg1$ and any  $r_0>0$ fixed, there holds
	\begin{equation}\label{B_1(x)Lf}
		\fint_{B_{r_0}(x)}u(y)\ud y=(-\alpha+o(1))\log|x|.
	\end{equation}
\end{lemma}
\begin{lemma}\label{lem:Q^+ and Q^-} (Lemma 2.10 in \cite{Li 23 Q-curvature})  Consider the normal solution $u(x)$ to \eqref{normal solution} with even integer $n\geq 2$.
	If  $Q^+$ has compact support, there holds
	$$ u(x)\leq -\alpha \log|x|+C,\quad |x|\gg1.$$
	Conversely, if $Q^-$ has compact support, there holds
	$$u(x)\geq -\alpha\log|x|-C,\quad |x|\gg1.$$
\end{lemma}
\begin{lemma}\label{lem:B_r_0|x|}
	(Lemma 2.5 in \cite{Li 23 Q-curvature}) Consider the normal solution $u(x)$ to \eqref{normal solution} with even integer $n\geq 2$.  For $|x|\gg1$ and any  $0<r_1<1$ fixed, there holds
	\begin{equation}\label{B_|x|/2v}
		\fint_{B_{r_1|x|}(x)}u(y)\ud y=(-\alpha +o(1))\log|x|.
	\end{equation}
\end{lemma}
\begin{lemma}\label{lem:B_x^-r_1}
(Lemma 2.6 in \cite{Li 23 Q-curvature}) Consider the normal solution $u(x)$ to \eqref{normal solution} with even integer $n\geq 2$.	For $|x|\gg1$ and  any   $r_2>0$ fixed, there holds
	\begin{equation}\label{B_|x|/2v}
		\fint_{B_{|x|^{-r_2}}(x)}u(y)\ud y=(-\alpha +o(1))\log|x|.
	\end{equation}
\end{lemma}

By  slightly modifying  Lemma 2.8 in \cite{Li 23 Q-curvature}, we  obtain the following property.
\begin{lemma}\label{lem: e^nu leq R-nalpha}
	Consider the normal solution $u(x)$ to \eqref{normal solution}. For $R\gg 1$ and  any $m>0$ fixed, there holds
	$$\fint_{B_{R}(0)\backslash B_{R-1}(0)}e^{mu}\ud x = R^{-m\alpha+o(1)}.$$
\end{lemma}

\begin{proof}
	On one hand, with help of Jensen's inequality and Lemma \ref{lem: B_1 L(f)},  for any $r_3>0$ fixed, one has
	\begin{equation}\label{e^Lf lower bound}
		\fint_{B_{r_3}(x)}e^{mu}\ud y\geq \exp\left(\fint_{B_{r_3}(x)}mu\ud y\right)=e^{(-m\alpha+o(1))\log|x|}.
	\end{equation}
	Now, we are going to deal with  the upper bound.
	For $|y|\gg1$, there holds
	$$|\int_{B_1(y)\backslash B_{1/4}(y)}\log\frac{1}{|y-z|}Q(z)e^{nu(z)}\ud z|\leq C\int_{B_1(y)\backslash B_{1/4}(y)}|Q(z)|e^{nu}\ud z\leq C.$$
	Combining with Lemma \ref{lem: L(f)}, we obtain
	$$u(y)=(-\alpha +o(1))\log|y|+\frac{2}{(n-1)!|\mathbb{S}^n|}\int_{B_{1/4}(y)}\log\frac{1}{|y-z|}Q(z)e^{nu(z)}\ud z.$$
	Then for $|x|\gg 1$ and $y\in B_{1/4}(x)$,
	one has
	\begin{equation}\label{L (f)leq -alpha log|x|}
		u(y)\leq (-\alpha +o(1))\log|x|+\frac{2}{(n-1)!|\mathbb{S}^n|}\int_{B_{1/2}(x)}\log\frac{1}{|y-z|}Q^+(z)e^{nu(z)}\ud z
	\end{equation}
	where we have used  the fact $|y-z|\leq 1$.
	
	Now, we claim that for $|x|\gg 1$ and $m>0$, there holds
	\begin{equation}\label{e^nLf leq |x|^n-alpha}
		\int_{B_{1/4}(x)}e^{mu(y)}\ud y\leq e^{(-m\alpha +o(1))\log|x|}.
	\end{equation}
	If $Q^+(z)=0$ a.e. on $B_{1/2}(x)$, we immediately obtain \eqref{e^nLf leq |x|^n-alpha} due to \eqref{L (f)leq -alpha log|x|}.
	Otherwise, Jensen's inequality yields that
	\begin{align*}
		&\int_{B_{1/4}(x)}e^{mu}\ud y\\
		\leq & |x|^{-m\alpha+o(1)}\int_{B_{1/4}(x)}\exp\left(\frac{2m}{(n-1)!|\mathbb{S}^n|}\int_{B_{1/2}(x)}\log\frac{1}{|y-z|}Q^+(z)e^{nu(z)}\ud z\right)\ud y\\
		\leq & |x|^{-m\alpha+o(1)}\int_{B_{1/4}(x)}\int_{B_{1/2}(x)}|y-z|^{-\frac{2m\|Q^+e^{nu}\|_{L^1(B_{1/2}(x))}}{(n-1)!|\mathbb{S}^n|}}\frac{Q^+(z)e^{nu(z)}}{\|Q^+e^{nu}\|_{L^1(B_{1/2}(x))}}\ud z\ud y.
	\end{align*}
	Since $Qe^{nu}\in L^1(\mr^n)$, there exists $R_2>0$ such that $|x|\geq R_2$, we have $$\|Q^+e^{nu}\|_{L^1(B_{1/2}(x))}\leq \frac{(n-1)!|\mathbb{S}^n|}{4m}.$$
	By applying Fuibini's theorem,  we prove
	the claim \eqref{e^nLf leq |x|^n-alpha}.
	
	For $r_3>0 $ fixed,  we could  choose finite balls $1\leq j\leq C(r_3)$ such that $B_{r_3}(x)\subset \cup_j B_{1/4}(x_j)$ with $x_j\in B_{r_3}(x)$.
	Hence, using the estimate \eqref{e^nLf leq |x|^n-alpha}, for $|x|\gg 1$, we have
	$$	\fint_{B_{r_3}(x)}e^{mu}\ud y\leq C\sum_{j}\fint_{B_{1/4}(x_j)}e^{mu}\ud y\leq  Ce^{(-m\alpha+o(1))\log|x|}=e^{(-m\alpha+o(1))\log|x|}
	$$
	Combing with \eqref{e^Lf lower bound}, one has
	\begin{equation}\label{e^mu for B_r}
		\log\fint_{B_{r_3}(x)}e^{mu}\ud y= (-m\alpha+o(1))\log|x|.
	\end{equation}
	By a direct computation, we obtain the inequality $C^{-1}R^{n-1} \leq |B_{R}(0)\backslash B_{R-1}(0)| \leq C R^{n-1}$ for $R \gg 1$, where $C$ is independent of $R$. We can select an index $C^{-1}R^{n-1} \leq i_R \leq C R^{n-1}$ such that the balls $B_{1/4}(x_j)$ with $|x_j| = R-\frac{1}{2}$ and $1 \leq j \leq i_R$ are pairwise disjoint, and the sum of the balls $B_4(x_j)$ cover the annulus $B_{R}(0)\backslash B_{R-1}(0)$. Applying the estimate  \eqref{e^mu for B_r},  we obtain the following result
	\begin{align*}
		\fint_{B_{R}(0)\backslash B_{R-1}(0)}e^{mu}\ud y\leq & \frac{1}{C^{-1} R^{n-1}}\sum_{j=1}^{i_R}\int_{B_4(x_j)}e^{mu}\ud y\\
		\leq & CR^{1-n}\sum_{j=1}^{i_R}|x_j|^{-n\alpha+o(1)}\\
		\leq &CR^{1-n}\cdot CR^{n-1}\cdot (R-\frac{1}{2})^{-m\alpha +o(1)}\\
		=&R^{-m\alpha +o(1)}.
	\end{align*}
	Similarly, there holds
	$$ \fint_{B_{R}(0)\backslash B_{R-1}(0)}e^{mu}\ud y \geq \frac{1}{C R^{n-1}}\sum_{j=1}^{i_R}\int_{B_{1/4}(x_j)}e^{mu}\ud y= R^{-m\alpha+o(1)}.$$
	
	Finally, we get the desired result.

\end{proof}	
\begin{lemma} \label{lem: e^mu leq R-malpha}
	
	For $R\gg1$ and $m>0$ fixed,
	there holds
	$$\int_{B_{\frac{3R}{2}}(0)\backslash B_{\frac{R}{2}}(0)}e^{mu}\ud x=R^{n-m\alpha+o(1)}.$$
\end{lemma}
\begin{proof}
	On one hand,	with help of Lemma \ref{lem: e^nu leq R-nalpha}, for any $\epsilon>0$, there exists $R_1>0$ such that $R\geq R_1$,
	\begin{equation}\label{int A_R e^nu}
		R^{n-1-m\alpha-\epsilon}\leq\int_{B_{R+1}\backslash B_R(0)}e^{mu}\ud x\leq R^{n-1-m\alpha+\epsilon}.
	\end{equation}
	Using the above estimate \eqref{int A_R e^nu},  for $R\geq 2R_1+2$, there holds
	\begin{align*}
		\int_{B_{\frac{3R}{2}}(0)\backslash B_{\frac{R}{2}}(0)}e^{mu}\ud x \leq& \sum^{[\frac{3R}{2}]}_{i=[\frac{R}{2}]}\int_{B_{i+1}(0)\backslash B_i(0)}e^{mu}\ud x\\
		\leq &\sum^{[\frac{3R}{2}]}_{i=[\frac{R}{2}]}i^{n-1-m\alpha+\epsilon}\\
		\leq & CR\cdot R^{n-1-m\alpha+\epsilon}\\
		\leq &CR^{n-m\alpha+\epsilon}
	\end{align*}
	
	On the other hand,  for $R\geq 2R_1+2$, we have
	\begin{align*}
		\int_{B_{\frac{3R}{2}}(0)\backslash B_{\frac{R}{2}}(0)}e^{mu}\ud x \geq& \sum^{[\frac{3R}{2}]-1}_{i=[\frac{R}{2}]+1}\int_{B_{i+1}(0)\backslash B_i(0)}e^{mu}\ud x\\
		\geq &\sum^{[\frac{3R}{2}]-1}_{i=[\frac{R}{2}]+1}i^{n-1-m\alpha_0-\epsilon}\\
		\geq & CR\cdot R^{n-1-m\alpha+\epsilon}\\
		=&CR^{n-m\alpha-\epsilon}.
	\end{align*}
	Thus we finish our proof.	
\end{proof}

The following lemma has also been established in \cite{Li 23 Q-curvature}. For readers' convenience, we give a brief proof here.
\begin{lemma}\label{lem: alpha  geq 1}
	(Theorem 2.18 in \cite{Li 23 Q-curvature}) Consider the normal solution $u(x)$  to \eqref{normal solution} with even integer $n\geq2$. Supposing  that the  volume is finite i.e.
	$e^{nu}\in L^1(\mr^n),$
	then there holds
	$$\int_{\mr^n}Qe^{nu}\ud x\geq \frac{(n-1)!|\mathbb{S}^n|}{2}.$$
\end{lemma}
\begin{proof}
	With help of Lemma \ref{lem:B_r_0|x|} and  $|x|\gg1$, Jensen's inequality implies that
	\begin{align*}
		&\int_{B_{|x|/2}(x)}e^{nu}\ud y\\
		\geq &|B_{|x|/2}(x)|\exp(\frac{1}{|B_{|x|/2}(x)|}\int_{B_{|x|/2}(x)}nu\ud y)\\
		=&C|x|^{n-n\alpha+o(1)}.
	\end{align*}
	Based on the assumption
	$e^{nu}\in L^1(\mr^n)$, letting $|x|\to\infty$,  there holds
	$\alpha\geq 1$
	i.e.
	$$\int_{\mr^n}Qe^{nu}\ud x\geq \frac{(n-1)!|\mathbb{S}^n|}{2}.$$
	Finally, we finish our proof.
\end{proof}

 Chang, Qing and Yang  \cite{CQY} firstly give a geometric criterion  about scalar curvature for normal solutions. More general results can be   found in  \cite{Wang-Wang} and \cite{Li 23 Q-curvature}.
\begin{lemma}\label{lema: normal}(Theorem 1.1 in \cite{Wang-Wang}, Theorem 4.1 in \cite{Li 23 Q-curvature})
	Given a smooth conformal metric $g=e^{2u}|dx|^2$ on $\mr^n$ where $n\geq 4$ is an even integer. The Q-curvature satisfies the equation \eqref{equ:-Delta ^n/2u=Ke^nu} with $Qe^{nu}\in L^1(\mr^n).$ If the scalar curvature $R_g$ defined by
	$$R_g=2(n-1)e^{-2u}(-\Delta u-\frac{n-2}{2}|\nabla u|^2)$$
	satisfies that $R_g\geq 0$ near infinity, then the solution to \eqref{equ:-Delta ^n/2u=Ke^nu} is normal.
\end{lemma}

With help of  above lemmas,  we are able to  prove Theorem \ref{thm: Q leq n-1 R geq0 half S^n}.

{\bf Prooof of Theorem \ref{thm: Q leq n-1 R geq0 half S^n}:}

Firstly, due the assumption $R_g\geq0$ near infinity, Lemma \ref{lema: normal} shows that $u(x)$ is normal.

	If $\int_{\mr^n}e^{nu}\ud x=+\infty$, it is trivial that \eqref{volume S^n/2} holds.	Now, we suppose that $e^{nu}\in L^1(\mr^n)$.
	On one hand, if $Q(x)$ is a constant, with help  of Lemma \ref{lem: alpha  geq 1},  $Q(x)$ must be a positive constant. Based on the classification theorem for normal solutions in \cite{Lin}, \cite{WX}, \cite{Mar MZ} or \cite{Xu05}, one has
	$$\int_{\mr^n}Qe^{nu}\ud x=(n-1)!|\mathbb{S}^n|.$$
	Our  assumption $Q\leq (n-1)!$ yields that
	$$\int_{\mr^n}e^{nu}\ud x\geq |\mathbb{S}^n|>\frac{|\mathbb{S}^n|}{2}.$$
	On the other hand, if $Q(x)$ is not a constant, using Lemma \ref{lem: alpha  geq 1} again, there holds
	$$\int_{\mr^n}e^{nu}\ud x>\frac{1}{(n-1)!}\int_{\mr^n}Qe^{nu}\ud x\geq \frac{|\mathbb{S}^n|}{2}.$$
	
	Finally, we finish our proof.

\section{Pohozaev identity}\label{section: Pohozaev identity}
The following Pohozaev-type  inequality is based on the work of Xu (See Theorem 2.1 in \cite{Xu05}).
\begin{lemma}\label{lem: Pohozaev for non-sign changing} Suppose that $u(x)$ is a normal solution to \eqref{normal solution} with $Q(x)$ doesn't change sign near infinity. Then there exists a sequence $R_i\to\infty$ such that
	$$\lim_{i\to \infty}\sup \frac{4}{n!|\mathbb{S}^n|}\int_{B_{R_i}(0)}x\cdot \nabla Q e^{nu}\ud x\leq \alpha(\alpha-2).$$
\end{lemma}
\begin{proof}
	By a direct computation, one has
	\begin{equation}\label{equ:x,nabla u}
		\langle x,\nabla u\rangle=-\frac{2}{(n-1)!|\mathbb{S}^n|}\int_{\mr^n}\frac{\langle x,x-y\rangle}{|x-y|^2}Q(y)e^{nu(y)}\ud y
	\end{equation}
	Multiplying by $Qe^{nu(x)}$ and integrating over the ball $B_R(0)$ for any $R>0$, we have
	\begin{equation}\label{equ:ingegrate x,nabla u}
		\int_{B_R(0)}Qe^{nu(x)}\left[-\frac{2}{(n-1)!|\mathbb{S}^n|}\int_{\mr^n}\frac{\langle x,x-y\rangle}{|x-y|^2}Q(y)e^{nu(y)}\ud y\right]\ud x=\int_{B_R(0)}Qe^{nu(x)}\langle x,\nabla u(x)\rangle\ud x.
	\end{equation}
	With $x=\frac{1}{2}\left((x+y)+(x-y)\right)$, for the left-hand side of \eqref{equ:ingegrate x,nabla u}, one has the following identity
	\begin{align*}
		LHS=&\frac{1}{2}\int_{B_R(0)}Qe^{nu(x)}\left[-\frac{2}{(n-1)!|\mathbb{S}^n|}\int_{\mr^n}Qe^{nu(y)}\ud y\right]\ud x   \\
		&+\frac{1}{2}\int_{B_R(0)}Qe^{nu(x)}\left[-\frac{2}{(n-1)!|\mathbb{S}^n|}\int_{\mr^n}\frac{\langle x+y,x-y\rangle}{|x-y|^2}Qe^{nu(y)}\ud y\right]\ud x.
	\end{align*}
	Now, we deal with  the last term of above equation by changing variables $x$ and $y$.
	\begin{align*}
		&	\int_{B_R(0)}Q(x)e^{nu(x)}\left[\int_{\mr^n}\frac{\langle x+y,x-y\rangle}{|x-y|^2}Q(y)e^{nu(y)}\ud y\right]\ud x\\
		=&\int_{B_R(0)}Q(x)e^{nu(x)}\left[\int_{\mr^n\backslash B_R(0)}\frac{\langle x+y,x-y\rangle}{|x-y|^2}Q(y)e^{nu(y)}\ud y\right]\ud x\\
		=&\int_{ B_{R/2}(0)}Q(x)e^{nu(x)}\left[\int_{\mr^n\backslash B_R(0)}\frac{\langle x+y,x-y\rangle}{|x-y|^2}Q(y)e^{nu(y)}\ud y\right]\ud x\\
		&+\int_{B_R(0)\backslash B_{R/2}(0)}Q(x)e^{nu(x)}\left[\int_{\mr^n\backslash B_{2R}(0)}\frac{\langle x+y,x-y\rangle}{|x-y|^2}Q(y)e^{nu(y)}\ud y\right]\ud x\\
		&+\int_{B_R(0)\backslash B_{R/2}(0)}Q(x)e^{nu(x)}\left[\int_{B_{2R(0)}\backslash B_{R}(0)}\frac{\langle x+y,x-y\rangle}{|x-y|^2}Q(y)e^{nu(y)}\ud y\right]\ud x\\
		=:&I_1(R)+I_2(R)+I_3(R).
	\end{align*} Notice that
	$$|I_1|\leq 3\int_{ B_{R/2}(0)}|Q(x)|e^{nu(x)}\ud x\int_{\mr^n\backslash B_R(0)}|Q(y)|e^{nu(y)}\ud y $$
	and
	$$|I_2|\leq 3\int_{ B_R(0)\backslash B_{R/2}(0)}|Q(x)|e^{nu(x)}\ud x\int_{\mr^n\backslash B_{2R}(0)}|Q(y)|e^{nu(y)}\ud y.$$
	Then both $|I_1|$ and $|I_2|$ tend to zero as $R\to \infty$ due to $Qe^{nu}\in L^1(\mr^n)$.
	Now,  we only need to  deal with  the term  $I_3$.
	Since $Q$ doesn't change sign near infinity, for $R\gg 1$,
	$$I_3(R)=\int_{B_R(0)\backslash B_{R/2}(0)}Q(x)e^{nu(x)}\left[\int_{B_{2R(0)}\backslash B_{R}(0)}\frac{|x|^2-|y|^2}{|x-y|^2}Q(y)e^{nu(y)}\ud y\right]\ud x\leq 0.$$
	As for the right-hand side of \eqref{equ:ingegrate x,nabla u}, by using divergence theorem, we have
	\begin{align*}
		RHS=&\frac{1}{n}\int_{B_R(0)}Q(x)\langle x,\nabla e^{nu(x)}\rangle\ud x   \\
		=&-\int_{B_R(0)}\left( Q(x)+\frac{1}{n}\langle x,\nabla Q(x)\rangle\right)e^{nu(x)}\ud x\\
		&+\frac{1}{n}\int_{\partial B_R(0)}Q(x)e^{nu(x)}R\ud \sigma.
	\end{align*}
	Since $Q(x)e^{nu(x)}\in L^1(\mr^n)$, there exist a sequence $R_i\to \infty$ such that
	$$\lim_{i\to\infty}R_i\int_{\partial B_{R_i}(0)}|Q|e^{nu}\ud\sigma=0.$$
	Otherwise,  there exists $\epsilon_0>0$ such that for large $R_1>1$ and any  $r\geq R_1$, there holds
	$|\int_{\partial B_r(0)}|Q|e^{nu}\ud \sigma|\geq \frac{\epsilon_0}{r}$	 and then
	$$|\int_{0}^R\int_{\partial B_r(0)}|Q|e^{nu}\ud\sigma\ud r|\geq -C+\int_{R_1}^R\frac{\epsilon_0}{r}\ud r\geq-C+\epsilon_0\log R$$
	which contradicts to $Qe^{nu}\in L^1(\mr^n)$.
	Thus there holds
	\begin{align*}
		\frac{1}{n}\int_{B_{R_i}(0)}\langle x\cdot \nabla Q\rangle e^{nu}\ud x=&-\int_{B_{R_i}(0)}Qe^{nu}\ud x+\frac{1}{n}R_i\int_{\partial B_{R_i}(0)}Qe^{nu}\ud \sigma
		+\frac{\alpha}{2}\int_{B_{R_i}(0)}Qe^{nu(x)}\ud x \\
		&+I_1(R_i)+I_2(R_i)+I_3(R_i)\\
		\leq &-\int_{B_{R_i}(0)}Qe^{nu}\ud x+\frac{1}{n}R_i\int_{\partial B_{R_i}(0)}Qe^{nu}\ud \sigma
		+\frac{\alpha}{2}\int_{B_{R_i}(0)}Qe^{nu(x)}\ud x \\
		&+I_1(R_i)+I_2(R_i)
	\end{align*}
	which yields that
	$$\lim_{i\to \infty}\sup \frac{4}{n!|\mathbb{S}^n|}\int_{B_{R_i}(0)}x\cdot \nabla Q e^{nu}\ud x\leq \alpha(\alpha-2).$$
	
\end{proof}

\begin{corollary}\label{cor: for n=2}
	Consider the  smooth conformally invariant   equation
	\begin{equation}\label{Gaussian equation}
		-\Delta u=Ke^{2u} \quad \mathrm{on}\;\;\mr^2
	\end{equation}
	with $Ke^{2u}\in L^1(\mr^2)$ and $e^{2u}\in L^1(\mr^2)$. Suppose that $x\cdot\nabla K\geq 0$ and $K(x)$ is non-negative near infinity. Then
	$$\int_{\mr^2}Ke^{2u}\ud x\geq 4\pi$$
	with $"="$ holds if and only if $K$ is a positive constant.
\end{corollary}
\begin{proof}
	With help of Theorem 2.2 in \cite{Li 23 Q-curvature},   the solution to \eqref{Gaussian equation} must be  a normal solution. Making use  of Lemma \ref{lem: Pohozaev for non-sign changing}, one has $\alpha\geq2$ i.e.
	$$\int_{\mr^2}Ke^{2u}\ud x\geq 4\pi.$$
	
	When the equality  achieves,  one has $x\cdot \nabla K=0$ a.e. which shows that $K(x)$ is a non-negative constant since $K$ is non-negative near infinity. If $K(x)\equiv 0$, since  $u(x)$ is normal,  we have  $u\equiv C$ which  contradicts to $e^{2u}\in L^1(\mr^2)$.  Hence, $K(x)$ must be  a positive constant.
	
	On the other hand, if $K(x)$ is a positive constant, with help of the classification theorem in  \cite{CL91}, one has
	$$\int_{\mr^2}Ke^{2u}\ud x=4\pi.$$
	Finally, we finish our proof.
\end{proof}
\begin{remark}
	By utilizing this result, we are able to partially answer  the question raised by Gui and Moradifam in Remark 5.1 of their paper \cite{GM}.
\end{remark}
With additional  assumptions, we are able to obtain the Pohozaev  identity.
\begin{lemma}\label{lem: Pohozaev identity}(See Lemma 2.1 in \cite{Li23})
	Consider a normal solution $u(x)$ to \eqref{normal solution}. Supposing that  $$|Q(x)|e^{nu}\leq C|x|^{-n}$$ near infinity, then  there exists  a sequence $R_i\to\infty$ such that
	$$\lim_{i\to\infty}\frac{4}{n!|\mathbb{S}^n|}\int_{B_{R_i}(0)}x\cdot \nabla Qe^{nu}\ud x=\alpha(\alpha-2).$$
\end{lemma}
\begin{proof}
	
	The proof is essentially the same as  Lemma \ref{lem: Pohozaev for non-sign changing}, except for the treatment of the term $I_3(R)$.
	Firstly, a direct computation yields that
	$$|I_3|\leq \int_{ B_R(0)\backslash B_{R/2}(0)}|Q(x)|e^{nu(x)}\int_{B_{2R(0)}\backslash B_{R}(0)}\frac{|x+y|}{|x-y|}|Q(y)|e^{nu(y)}\ud y\ud x.$$
	For each $x\in B_R(0)\backslash B_{R/2}(0)$ and $R\gg1$, based on the assumption $|Q|e^{nu(x)}\leq C|x|^{-n}$ near infinity, a direct computation yields that
	\begin{align*}
		&\int_{B_{2R(0)}\backslash B_{R}(0)}\frac{| x+y|}{|x-y|}|Q(y)|e^{nu(y)}\ud y\\
		\leq &CR^{1-n}\int_{B_{2R(0)}\backslash B_{R}(0)}\frac{1}{|x-y|}\ud y\\
		\leq &CR^{1-n}\int_{B_{3R}(0)}\frac{1}{|y|}\ud y\leq C.
	\end{align*}
	Thus we obtain that
	$$|I_3|\leq C\int_{ B_R(0)\backslash B_{R/2}(0)}|Q|e^{nu}\ud x\to 0,\;\mathrm{as}\; R\to\infty.$$
	Continuing along the same line of reasoning as presented in Lemma \ref{lem: Pohozaev for non-sign changing}, we ultimately demonstrate the existence of a sequence $R_i\to\infty$ such that
	$$\lim_{i\to\infty}\frac{4}{n!|\mathbb{S}^n|}\int_{B_{R_i}(0)}x\cdot \nabla Q e^{nu}\ud x= \alpha(\alpha-2).$$

\end{proof}

\begin{lemma}\label{lem: PH for radial  symmetric}
	Consider a normal solution $u(x)$ to \eqref{normal solution} and $n\geq4$. Supposing  that $u(x)$ is radially symmetric, then there exists  a sequence $R_i\to\infty$ such that
	$$\lim_{i\to\infty}\frac{4}{n!|\mathbb{S}^n|}\int_{B_{R_i}(0)}x\cdot \nabla Qe^{nu}\ud x=\alpha(\alpha-2).$$
\end{lemma}
\begin{proof}
	Due to the same reason as before, we only need to deal with the  term $I_3(R)$ in Lemma \ref{lem: Pohozaev for non-sign changing}.
	With help of \eqref{equ:-Delta ^n/2u=Ke^nu}, it is easy to see that $Q(x)=e^{-nu(x)}(-\Delta)^{\frac{n}{2}}u(x)$ is also radially symmetric. infinity.
	For brevity, we set $\varphi(x):=Qe^{nu}$ which is radially symmetric. For $r=|x|\gg1$, we simply write $\varphi(x)$ as $\varphi(r)$ if needed.
	 Hence for $R\gg1 $, one has
	 \begin{equation}\label{varphi (r) to 0}
	 	\int_{R/2}^{R}r^{n-1}|\ms^{n-1}|\cdot|\varphi (r)|\ud r=\int_{B_{R}(0)\backslash B_{R/2}(0)}|Q|e^{nu}\ud x\to 0
	 \end{equation}
	as $R\to\infty$ due to $Qe^{nu}\in L^1(\mr^n).$
	For $n\geq 4$, by using the fact the $|x-y|^{2-n}$ is the Green's function for Laplacian on $\mr^n$ and mean value property for harmonic function, one may check that for $r=|x|<|y|$
	$$\frac{1}{|\partial B_r(0)|}\int_{\partial B_r(0)}\frac{1}{|x-y|^{n-2}}\ud\sigma(x)=\frac{1}{|y|^{n-2}}.$$
	Then for any $y\in B_{2R}(0)\backslash B_{R}(0)$, one has
	\begin{align*}
	|I_3(R)|=&\left|\int_{B_{2R}(0)\backslash B_R(0)}Q(y)e^{nu(y)}\int_{B_{R}(0)\backslash B_{R/2}(0)}\frac{|x|^2-|y|^2}{|x-y|^2}Q(x)e^{nu(x)}\ud x\ud y\right|\\
	\leq &\int_{B_{2R}(0)\backslash B_R(0)}|Q(y)e^{nu(y)}|	\int_{R/2}^{R}|\varphi(r)|(|y|^2-r^2)(|y|+r)^{n-4}\int_{\partial B_r(0)}\frac{1}{|x-y|^{n-2}}\ud\sigma(x)\ud r\ud y\\
	=&\int_{B_{2R}(0)\backslash B_R(0)}|Q(y)e^{nu(y)}|	\int_{R/2}^{R}|\varphi(r)|(|y|^2-r^2)(|y|+r)^{n-4}\frac{|\partial B_r(0)|}{|y|^{n-2}}\ud r\ud y\\
	\leq& C \int_{B_{2R}(0)\backslash B_R(0)}|Q(y)e^{nu(y)}|\ud y\cdot \int_{R/2}^{R}|\varphi(r)|r^{n-1}\ud r.
	\end{align*}
By using \eqref{varphi (r) to 0} and the fact $Qe^{nu}\in L^1(\mr^n)$, we show that
$$I_3(R)\to 0 ,\quad \mathrm{as}\; R\to 0.$$
Thus we finish our proof.
\end{proof}

\section{Polynomial cone condition on Q-curvature}\label{section: PCC condition}
\begin{definition}\label{def:s-cone}
	We say a function $\psi(x)\in L^\infty_{loc}(\mr^n)$ satisfying {\bf s-cone}  condition if
	there exists $s\in \mathbb{R}$,  $0<r_0<1$  such that
	$$|\psi(x)|\leq C(|x|+1)^s$$ and
	a sequence  $\{x_i\}\subset\mr^n$ with $|x_i|\geq 1$
	and $|x_i|\to\infty$ as $i\to\infty$ such that  for each $i$ and
	$x\in B_{r_0|x_i|}(x_i)$ there holds
	$$ \frac{|\psi(x)|}{|x|^s}\geq c_1>0$$
	where $c_1$ is a constant independent  of $i$.
\end{definition}

The definition mentioned here is derived from  \cite{Li23} where the first author  focused on scenarios where $Q(x)$ is a polynomial, which is a common case that meets the {\bf s-cone} condition.

\begin{lemma}\label{lem:s-cone}
	Each non-constant polynomial $P(x)$ on $\mr^n$ satisfies {\bf s-cone} condition with $s=\deg P$.
\end{lemma}
\begin{proof}
	We can decompose the non-constant polynomial $P(x)$ as
	$$P(x)=H_{s}(x)+P_{s-1}(x)$$
	where $H_{s}(x)$ is a homogeneous function with degree equal to $s\geq1$ and $P_{s-1}(x)$ is a polynomial with degree at most $s-1$.
	Immediately, one has
	$$|P(x)|\leq C(|x|+1)^s.$$
	We choose  the polar coordinate such that  $x=\xi(r,\theta)$ with $r\geq 0$ and $\theta\in\mathbb{S}^{n-1}$. Then one has
	$$H_s(x)=r^{s}\varphi_s(\theta)$$ where $\varphi_s(\theta)$ is a non-zero smooth function defined on $\mathbb{S}^{n-1}$. There exist $c_1>0$ and a geodesic ball $\hat B_{s_1}(\theta_0)\subset\mathbb{S}^{n-1}$ such that
	$$|\varphi_s(\theta)|\geq 2c_1>0$$
	for any $\theta\in \hat B_{s_1}(\theta_0)$.  Then
	$$|P(x)|\geq2c_1|x|^s-c_2|x|^{s-1}$$
	where $c_2>0$ is a constant depending only on the coefficients of $P_{s-1}(x)$. We choose $R_1=\max\{1,\frac{c_2}{c_1}\}$ and then one has
	$$|P(x)|\geq c_1|x|^s$$
	for $(r,\theta)\in [R_1,+\infty)\times \hat B_{s_1}(\theta_0)$.
	It is not hard to see that  there exist $x_0\in\mr^n$ with $|x_0|=1$ and $0<s_0<1$ such that
	$$ \xi^{-1}(B_{s_0}(x_0))\subset [0,+\infty)\times \hat B_{s_1}(\theta_0).$$
	Meanwhile, one may check that for any $t>0$
	$$B_{s_0t}(tx_0)=tB_{s_0}(x_0)$$
	where $tB_{s_0}(x_0):=\{tx\in \mr^n|x\in B_{s_0}(x_0)\}.$
	For any $x\in B_{s_0t}(tx_0)$, there holds
	$|x|\geq (1-s_0)t.$
	Then there exists $t_1>0$ such that $t\geq t_1$
	$$\xi^{-1}(B_{s_0t}(tx_0))\subset [R_1,+\infty)\times \hat B_{s_1}(\theta_0).$$
	In particular, for any $t\geq t_1$ and $x\in B_{ts_0}(tx_0)$ one has
	$$|P(x)|\geq c_1|x|^s.$$
	Thus $P(x)$ satisfies {\bf s-cone} condition with $s=\deg P$.
	
\end{proof}

\begin{lemma}\label{lem:alpha geq tau Q}
	Consider the   normal solution $u(x)$ to  \eqref{normal solution} with  $Q(x)$ satisfying {\bf s-cone} condition. Then
	$$\alpha\geq 1+\frac{s}{n}.$$
\end{lemma}
\begin{proof}
	Due to $Q(x)$ satisfying s-cone condition, there exits a sequence $\{x_i\}$ and $0<s_0<1$ such that in each ball $x\in B_{s_0|x_i|}(x_i)$
	$$ |Q(x)|\geq C|x|^s$$
	With help of Lemma \ref{lem:B_r_0|x|} and Jensen's inequality, one has
	\begin{align*}
		\int_{B_{s_0|x_i|}(x_i)}|Q|e^{nu}\ud x\geq &C\int_{B_{s_0|x_i|}(x_i)}|x|^se^{nu}\ud x\\
		\geq &C|x_i|^s\int_{B_{s_0|x_i|}(x_i)}e^{nu}\ud x\\
		\geq &C|x_i|^{s}|B_{s_0|x_i|}(x_i)|\exp\left(\fint_{B_{s_0|x_i|}(x_i)}nu\ud x\right)\\
		\geq &C|x_i|^{s+n-n\alpha+o(1)}.
	\end{align*}
	Due to $Qe^{nu}\in L^1(\mr^n)$, letting $i\to\infty$, we have
	$$\alpha\geq 1+\frac{s}{n}.$$
	
\end{proof}

\begin{lemma}\label{lem: Qe^nu leq  C|x|^-n}
	Consider the normal solution $u(x)$ to  \eqref{normal solution} with  $Q(x)$ satisfying {\bf s-cone} condition. Supposing that  there exists $ s_1<s$ such that  $Q^+\leq C |x|^{s_1}$ or $Q^-\leq C|x|^{s_1}$ near infinity, then   there holds
	$$|Q(x)|e^{nu}\leq C|x|^{-n}$$
	near infinity.
\end{lemma}
\begin{proof}
	By a direct computation, we have
	\begin{align*}
		&\frac{(n-1)!|\mathbb{S}^n|}{2}(u(x)+\alpha\log|x|)\\
		=&\int_{\mr^n}\log\frac{|x|\cdot(|y|+1)}{|x-y|}Qe^{nu}\ud y+\int_{\mr^n}\log\frac{|y|}{|y|+1}Qe^{nu}\ud y+C\\
		=&\int_{\mr^n}\log\frac{|x|\cdot(|y|+1)}{|x-y|}Q^+e^{nu}\ud y-\int_{\mr^n}\log\frac{|x|\cdot(|y|+1)}{|x-y|}Q^-e^{nu}\ud y+C\\
		=&:I_1-II_2+C
	\end{align*}
	For $|x|\geq 1 $, it is easy to check that
	$$\frac{|x|\cdot(|y|+1)}{|x-y|}\geq 1$$
	which shows that
	$$\log\frac{|x|\cdot(|y|+1)}{|x-y|}\geq 0$$
	and then  $I_1\geq 0$, $II_2\geq0$.
	
	From now on, we suppose that $|x|\gg1$.
	
	If $Q^+(x)\leq C|x|^{s_1}$ near infinity, we split $I_1$ as follows
	$$I_1=\int_{|x-y|\leq \frac{|x|}{2}}\log\frac{|x|\cdot(|y|+1)}{|x-y|}Q^+e^{nu}\ud y+\int_{|x-y|\geq \frac{|x|}{2}}\log\frac{|x|\cdot(|y|+1)}{|x-y|}Q^+e^{nu}\ud y=:I_{1,1}+I_{1,2}$$
	Using the estimate \eqref{e^mu for B_r} and Lemma \ref{lem: e^mu leq R-malpha}, a direct computation  and H\"older's inequality yield that
	\begin{align*}
		I_{1,1}=& \int_{|x-y|\leq  \frac{|x|}{2}}\log(|x|(|y|+1))Q^+e^{nu}\ud y+ \int_{|x-y|\leq  \frac{|x|}{2}}\log\frac{1}{|x-y|}Q^+e^{nu}\ud y\\
		\leq &\int_{|x-y|\leq  \frac{|x|}{2}}\log(|x|(|y|+1))Q^+e^{nu}\ud y+ \int_{|x-y|\leq  1}\log\frac{1}{|x-y|}Q^+e^{nu}\ud y\\
		\leq &C\log(4|x|^2)|x|^{s_1}\int_{|x-y|\leq \frac{|x|}{2}}e^{nu}\ud y+C|x|^{s_1}\int_{|x-y|\leq 1}\log\frac{1}{|x-y|}e^{nu}\ud y\\
		\leq &C\log(4|x|^2)|x|^{s_1}\int_{|x-y|\leq \frac{|x|}{2}}e^{nu}\ud y+C|x|^{s_1}(\int_{|x-y|\leq 1}(\log\frac{1}{|x-y|})^2\ud y)^{\frac{1}{2}}(\int_{|x-y|\leq1}e^{2nu}\ud y)^{\frac{1}{2}}\\
		\leq& C\log(4|x|^2)|x|^{s_1}\int_{\frac{|x|}{2}\leq |y|\leq \frac{3|x|}{2}}e^{nu}\ud y+C|x|^{s_1}|x|^{-n\alpha+o(1)}\\
		\leq &C\log(4|x|^2)|x|^{s_1}|x|^{n-n\alpha+o(1)}+C|x|^{s_1}|x|^{-n\alpha+o(1)}
	\end{align*}
	With help of  Lemma   \ref{lem:alpha geq tau Q} and $s>s_1$,  there holds
	$$I_{1,1}\leq C\log(4|x|^2)|x|^{s_1-s+o(1)}+C|x|^{-n+s_1-s+o(1)}$$
	which shows that
	\begin{equation}\label{I_1,1 leq C}
		I_{1,1}\in L^\infty(\mr^n).
	\end{equation}
	With help of Lemma \ref{lem: e^nu leq R-nalpha} and Lemma \ref{lem:alpha geq tau Q}, choosing $\epsilon=\frac{s-s_1}{2}$,
	there exist  an integer $R_1>0$ such that for any $i\geq R_1+1$
	$$\int_{B_{i}(0)\backslash B_{i-1}(0)}e^{nu}\ud x\leq i^{n-1-n\alpha+\epsilon}\leq i^{-1-s+\epsilon}.$$
	As for the second term $I_{1,2}$,  for small $\epsilon>0$,
	Then there holds
	\begin{align*}
		I_{1,2}\leq& \int_{|x-y|\geq \frac{|x|}{2}}\log2(|y|+1)Q^+e^{nu}\ud y\\
		\leq &\int_{\mr^n}\log(2|y|+2)Q^+e^{nu}\ud y\\
		\leq &C +C\lim_{k\to\infty}\sum_{i=R_1+1}^k\log(2i+2)i^{s_1}\int_{B_i(0)\backslash B_{i-1}(0)}e^{nu}\ud y\\
		\leq&C+C\lim_{k\to\infty}\sum_{i=R_1+1}^k\log(2i+2)i^{s_1}i^{-1-s+\epsilon}<+\infty.
	\end{align*}
	Combing with \eqref{I_1,1 leq C}, one has
	$$I_1\in L^\infty(\mr^n).$$
	Then there holds
	\begin{equation}\label{u leq -alpha log |x|+C}
		u(x)\leq -\alpha\log|x|+C.
	\end{equation}
	Since $Q(x)$ satisfies  {\bf s-cone} condition, Lemma \ref{lem:alpha geq tau Q} and \eqref{u leq -alpha log |x|+C} imply that
	$$|Q(x)|e^{nu}\leq C|x|^s\cdot|x|^{-n\alpha}\leq C|x|^{-n}.$$
	
	On the other hand,  if $Q^-\leq C|x|^{s_1}$ near infinity, using similar argument, one has $II_2\leq C$ which yields that
	$$u(x)\geq -\alpha\log|x|-C.$$
	Due to  $Q(x)$ satisfying {\bf s-cone} condition, there holds
	\begin{align*}
		\int_{B_{s_0|x_i|}(x_i)}|Q|e^{nu}\ud x\geq &C\int_{B_{s_0|x_i|}(x_i)}|x|^se^{nu}\ud x\\
		\geq &C|x_i|^s\int_{B_{s_0|x_i|}(x_i)}e^{nu}\ud x\\
		\geq &C|x_i|^{s+n-n\alpha}
	\end{align*}
	which yields that
	$$\alpha>1+\frac{s}{n}.$$
	Since $|Q(x)|\leq C|x|^s$ near infinity and $\alpha>1+\frac{s}{n}$,
	similar argument yields that  $I_1\in L^\infty(\mr^n)$. Finally, we obtain  that
	$$|u+\alpha\log|x||\leq C.$$
	Thus  one has
	$$|Q(x)|e^{nu}|x|^n\leq C |x|^{s+n-n\alpha}=o(1).$$
	
	Finally, we finish our proof.
\end{proof}
As an application, we generalize Corollary 2.4 in \cite{Li23}.
\begin{theorem}
	Suppose that a smooth function $Q(x)$ satisfying {\bf s-cone} condition with $s\geq n$ as well as  $$x\cdot \nabla Q\leq 0.$$ There is no normal solution $u(x)$ to \eqref{normal solution} on $\mr^n$ with $Qe^{nu}\in L^1(\mr^n)$ where  even integer $n\geq2$.
\end{theorem}
\begin{proof}
	We argue by contradiction. Assume that such solution $u(x)$ exists.
	Since $x\cdot \nabla Q\leq 0$, we have $Q\leq C.$  Based on $Q$ satisfying {\bf s-cone} condition with $s\geq n$, with help of Lemma \ref{lem: Qe^nu leq  C|x|^-n} and Lemma \ref{lem: Pohozaev identity}, there exists a sequence $R_i\to\infty$ such that
	$$\lim_{i\to\infty}\frac{4}{n!|\mathbb{S}^n|}\int_{B_{R_i}(0)}x\cdot \nabla Qe^{nu}\ud x=\alpha(\alpha-2).$$
	Due to $x\cdot \nabla Q\leq 0$, we further have
	$$\frac{4}{n!|\mathbb{S}^n|}\int_{\mr^n}x\cdot \nabla Qe^{nu}\ud x=\alpha(\alpha-2)$$
	which yields that
	\begin{equation}\label{alpha bound}
		0<\alpha<2
	\end{equation}
	since $Q$ is obviously not a constant.
	However, Lemma \ref{lem:alpha geq tau Q} yields that
	$$\alpha\geq 1+\frac{s}{n}\geq 2$$
	which contradicts to \eqref{alpha bound}.
	
\end{proof}

In particular, suppose that $Q(x)$ is a non-constant  polynomial splitting as
$$Q(x)=H_{m}(x)+P_{m-1}(x)$$
where $H_{m}$ is a homogeneous function and $P_{m-1}(x)$ is a polynomial of degree at most $m-1$.
If either  $H_m(x)\geq 0$ or $H_m(x)\leq 0$, then we have
either $Q^-\leq C|x|^{m-1}$ or $Q^+\leq C|x|^{m-1}$ near infinity respectively.

\section{Bol's inequality for polynomial Q-curvature}\label{section: Bol's inequality for polynomial}
For readers' convenience,  we establish the modified Ding's lemma (See Lemma 1.1 in \cite{CL91}). It has  also been established in Proposition 8.5 of \cite{GL} or Lemma 2.6 of  \cite{LX23}.
\begin{lemma}\label{lem: modified Ding's lemma}
	Consider a smooth solution  $u(x)$ to the following equation
	$$-\Delta u=fe^{2u}\quad\mathrm{on}\;\; \mathbb{R}^2$$
	with smooth function $f\leq 1$, then there holds
	$$\int_{\mathbb{R}^2}e^{2u}\ud x\geq 4\pi$$
	when  the equality achieves, one has $f\equiv1$.
\end{lemma}
In this section, we are going to deal with polynomial Q-curvature.
\begin{theorem}\label{thm: bol's inequality for R4}
	Consider  $u(x)$ is a normal  solution to \eqref{normal solution} on $\mr^4$ where   $Q(x)\leq 6$ is a polynomial. Then
	there holds
	$$\int_{\mr^4}e^{4u}\ud x \geq |\mathbb{S}^4|$$
	with equality holds if and only if $Q(x)\equiv6$.
\end{theorem}
\begin{proof}
	
	We only need to deal with the case $e^{4u}\in L^1(\mr^4)$.
	
	Firstly,	if $Q(x)$ is a constant, Lemma \ref{lem: alpha  geq 1} deduces  that $Q(x)$ must be a positive constant. The classification theorem in \cite{Lin} shows that
	$$\int_{\mr^4}Qe^{4u}\ud x=6|\mathbb{S}^4|$$
	which yields that
	$$\int_{\mr^4}e^{4u}\ud x\geq |\mathbb{S}^4|$$
	with $"="$ holds if and only if $Q\equiv 6$.
	
	Secondly, if  $Q(x)$ is a non-constant  polynomial with $Q(x)\leq 6$, then the degree of $Q(x)$ must be an even integer.
	Lemma \ref{lem:s-cone} yields that $Q(x)$ satisfying {\bf s-cone} condition with $s=\deg Q$.
	If $\deg Q\geq 4$, Lemma \ref{lem:alpha geq tau Q} concludes that
	$$\int_{\mr^4}Qe^{4u}\ud x\geq 6|\mathbb{S}^4|.$$ Since non-constant $Q(x)\leq 6$,  we obtain that
	$$\int_{\mr^4}e^{4u}\ud x>|\mathbb{S}^4|.$$
	
	Finally, the  remaining case is $\deg Q(x)=2$. Since $Q(x)\leq 6$, up to a rotation and a translation of the coordinates, we may suppose that
	$$Q(x)=a+\sum_{i=1}^4a_ix_i^2$$
	where $a_i\leq 0$, $\sum_{i=1}^4a_i^2\not=0$ and $a\leq 6$. Using $e^{4u}\in L^1(\mr^4)$ and $Qe^{4u}\in L^1(\mr^4)$, one has
	$$\sum^4_{i=1}a_ix_i^2e^{4u}\in L^1(\mr^4)$$ which yields that $x\cdot\nabla Qe^{4u}\in L^1(\mr^4)$.
	Applying Lemma \ref{lem: Pohozaev identity} and  Lemma \ref{lem: Qe^nu leq  C|x|^-n}, there holds
	\begin{equation}\label{equ: PH for R^4}
		\alpha(\alpha-2)=\frac{1}{16\pi^2}\int_{\mr^4}x\cdot \nabla Q e^{4u}\ud x\\
		=\frac{1}{8\pi^2}\int_{\mr^4}\sum^4_{i=1}a_ix_i^2e^{4u}\ud x<0.
	\end{equation}
	Combing \eqref{equ: PH for R^4} with Lemma \ref{lem:alpha geq tau Q}, one has
	\begin{equation}\label{alpha bound for R4}
		\frac{3}{2}\leq \alpha<2.
	\end{equation}
	Immediately, one has $a>0.$
	Using the identity \eqref{equ: PH for R^4} again, one has
	$$\int_{\mr^4}e^{4u}\ud x=\frac{8\pi^2}{a}\alpha(3-\alpha).$$
	Making using of the estimate \eqref{alpha bound for R4} and  $a\leq 6$, there holds
	$$\int_{\mr^4}e^{4u}\ud x>\frac{8\pi^2}{3}=|\mathbb{S}^4|.$$
	
	Finally, we finish our proof.
	
\end{proof}

For $n\geq 6$, due to  our technical constraints, we need introduce additional assumptions on the polynomial $Q$.
\begin{theorem}\label{thm: bol's inequality for n geq 6}
	Consider the normal solution $u(x)$ to \eqref{normal solution} on $\mr^n$  where even integer $n\geq 6$ with  $Q(x)=(n-1)!+\varphi(x)$ where $\varphi(x)\leq 0$ is a polynomial. If  $4\leq \deg\varphi \leq n-2$, we additionally  assume that
	\begin{equation}\label{homogeneous property}
		x\cdot \nabla \varphi(x)\geq (\deg\varphi )\varphi(x)
	\end{equation}
	up to a rotation or a translation of the coordinates.
	Then there holds
	\begin{equation}\label{voulme for R6}
		\int_{\mr^n}e^{nu}\ud x\geq |\mathbb{S}^n|
	\end{equation}
	with equality  holds if and only if $\varphi\equiv 0.$
\end{theorem}

\begin{proof}
	If $\int_{\mr^n}e^{nu}\ud x=+\infty$, it is trvial that the estimate \eqref{voulme for R6} holds . We just need to consider the case $e^{nu}\in L^1(\mr^n)$. Then one has
	$$\int_{\mr^n}|\varphi|e^{nu}\ud x\leq (n-1)!\int_{\mr^n}e^{nu}\ud x+\int_{\mr^n}|(n-1)!+\varphi(x)|e^{nu}\ud x<+\infty.$$
	If $\varphi$ is a constant, with help of Lemma \ref{lem: alpha  geq 1} and the the classification theorem in \cite{WX}, \cite{Mar MZ} or \cite{Xu05}, we must have $\varphi>-(n-1)!$ and
	$$\int_{\mr^n}e^{nu}\ud x=\frac{(n-1)!}{(n-1)!+\varphi} |\mathbb{S}^n|\geq|\mathbb{S}^n|$$
	with equality holds if and only if $\varphi\equiv0.$
	
	Now, we are going to deal with the non-constant case. Due to $\varphi\leq 0$, the degree of $\varphi$ must be an even integer.
	If $\deg\varphi=2$, following the same argument in the proof of Theorem \ref{thm: bol's inequality for R4}, there holds
	$$\int_{\mr^n}e^{nu}\ud >|\mathbb{S}^n|.$$
	When $\deg(\varphi)\geq n$, with help of Lemma \ref{lem:alpha geq tau Q}, one has $\alpha\geq 2$.
	Then there holds
	$$\int_{\mr^n}e^{nu}\ud x> \int_{\mr^n}\frac{(n-1)!+\varphi}{(n-1)!}e^{nu}\ud x\geq |\mathbb{S}^n|.$$
	If  $4\leq \deg\varphi <n$ and $\alpha\geq 2$, due to  the same reason,  the estimate \eqref{voulme for R6} still holds.
	Thus the remaining case is $4\leq \deg\varphi\leq n-2$ and $\alpha<2$. With help of Lemma \ref{lem:alpha geq tau Q}, there holds
	\begin{equation}\label{alpha bound for n geq 6}
		1+	\frac{\deg\varphi}{n}\leq \alpha<2.
	\end{equation}
	Making use  of Lemma \ref{lem: Pohozaev identity}, there exists a sequence  $R_i\to\infty$  as $i\to\infty$ such that
	$$\alpha(\alpha-2)=\frac{4}{n!|\mathbb{S}^n|}\lim_{i\to\infty}\int_{B_{R_i}(0)}x\cdot \nabla \varphi e^{nu}\ud x.$$
	Using the assumption $x\cdot\nabla\varphi\geq \deg(\varphi)\varphi,$ one has
	$$\int_{B_{R_i}(0)}x\cdot \nabla \varphi e^{nu}\ud x\geq \deg(\varphi)\int_{B_{R_i}(0)}\varphi e^{nu}\ud x.$$
	Then a direct computation yields that
	\begin{align*}
		\alpha(\alpha-2)\geq &\frac{4}{n!|\mathbb{S}^n|} \deg(\varphi)\int_{\mr^n}\varphi e^{nu}\ud x\\
		=&\frac{2\deg (\varphi)}{n}\alpha-\frac{4\deg\varphi}{n|\mathbb{S}^n|}\int_{\mr^n}e^{nu}\ud x
	\end{align*}
	which shows that
	$$\int_{\mr^n}e^{nu}\ud x\geq \frac{n|\mathbb{S}^n|}{4\deg\varphi}\alpha(2+\frac{2\deg\varphi}{n}-\alpha).$$
	Using \eqref{alpha bound for n geq 6}, one has
	$$\int_{\mr^n}e^{nu}\ud x>|\mathbb{S}^n|.$$
	
	Finally, we finish our proof.
\end{proof}

\section{Bol's inequality for radial solutions}\label{section: Bol's inequality for radial solution}
\begin{theorem}\label{thm: radial solution for Q leq (n-1)!}
	Supposing that $u(x)$ is a radially symmetric and  normal  solution  to \eqref{normal solution} with $$Q(x)\leq (n-1)!$$ where even integer $n\geq 4$. Then there holds
	$$\int_{\mr^n}e^{nu}\ud x\geq|\mathbb{S}^n|$$
	with equality holds if and only if $Q\equiv(n-1)!$.
\end{theorem}
\begin{proof}
	We just need to deal with the case $e^{nu}\in L^1(\mr^n)$.  Set
	\begin{equation}\label{normal v}
		v(x):=\frac{2}{(n-1)!|\mathbb{S}^n|}\int_{\mr^n}\log\frac{|y|}{|x-y|}(n-1)!e^{nu(y)}\ud y.
	\end{equation}
	Firstly, we claim that $v(x)$ is also radial symmetric. For any rotation $T$, we have $u(Ty)=u(y)$ due to $u(x)$ is radial symmetric. In fact, by rotating the coordinates, there holds
	\begin{align*}
		v(Tx)=&\frac{2}{(n-1)!|\mathbb{S}^n|}\int_{\mr^n}\log\frac{|y|}{|Tx-y|}(n-1)!e^{nu(y)}\ud y\\
		=&\frac{2}{(n-1)!|\mathbb{S}^n|}\int_{\mr^n}\log\frac{|Ty|}{|Tx-Ty|}(n-1)!e^{nu(Ty)}\ud y\\
		=&\frac{2}{(n-1)!|\mathbb{S}^n|}\int_{\mr^n}\log\frac{|y|}{|x-y|}(n-1)!e^{nu(y)}\ud y\\
		=&v(x).
	\end{align*}
Setting  $h:=u-v$.
For $n\geq 4$,
a direct computation yields that
\begin{equation}\label{Delat h geq 0}
	\Delta h=\frac{2(n-2)}{(n-1)!|\mathbb{S}^n|}\int_{\mr^n}\frac{1}{|x-y|^2}\left((n-1)!-Q\right)e^{nu}\ud y\geq 0
\end{equation}
 and there holds
$$(-\Delta)^{\frac{n}{2}}v=(n-1)!e^{nh}e^{nv}.$$
Combining with \eqref{normal v}, we find that $v$ is a normal solution with Q-curvature equal to $(n-1)!e^{nh}$.
For brevity, we denote the normalized integrated Q-curvature as
$$\alpha_0:=\frac{2}{(n-1)!|\mathbb{S}^n|}\int_{\mr^n}(n-1)!e^{nh}e^{nv}\ud x=\frac{2}{|\mathbb{S}^n|}\int_{\mr^n}e^{nu}\ud x.$$
By using Lemma \ref{lem: PH for radial  symmetric}, there exists a sequence $R_i\to\infty$ such that
\begin{equation}\label{Ph inequa}
	\lim_{i\to \infty}\sup\frac{4}{n!|\mathbb{S}^n|}\int_{B_{R_i}(0)}x\cdot \nabla ((n-1)!e^{nh})e^{nv}\ud x= \alpha_0(\alpha_0-2).
\end{equation}
Using divergence theorem and the condition that  $u(x)$ is radially symmetric, there holds
\begin{align*}
	&\frac{4}{n!|\mathbb{S}^n|}\int_{B_{R_i}(0)}x\cdot \nabla ((n-1)!e^{nh})e^{nv}\ud x\\
	=&\frac{4}{|\mathbb{S}^n|}\int_{B_{R_i}(0)}x\cdot \nabla h e^{nu}\ud x\\
	=&\frac{4}{|\mathbb{S}^n|}\int_0^{R_i}\int_{\mathbb{S}^{n-1}(r)}r\frac{\partial h}{\partial r} e^{nu}\ud\sigma\ud r\\
	=&\frac{4}{|\mathbb{S}^n|}\int_0^{R_i} re^{nu(r)}\int_{\mathbb{S}^{n-1}(r)}\frac{\partial h}{\partial r} \ud\sigma\ud r\\
	=&\frac{4}{|\mathbb{S}^n|}\int_0^{R_i}re^{nu(r)}\int_{B_r(0)}\Delta h\ud x\ud r\\
	\geq &0
\end{align*}
where the last term comes from \eqref{Delat h geq 0}.
Finally, the estimate \eqref{Ph inequa} yields that
$$\alpha_0\geq 2$$
i.e.
$$\int_{\mr^n}e^{nu}\ud x\geq |\mathbb{S}^n|$$
when the  equality achieves, one has
$\Delta h\equiv 0$ which yields that
$Q\equiv (n-1)!$.

On the other hand, if $Q\equiv (n-1)!$,  using  Lemma \ref{lem: Pohozaev identity} and Lemma \ref{lem: Qe^nu leq  C|x|^-n}, we also obtain that $\alpha_0=2$.

\end{proof}

\begin{theorem}\label{thm: radial solution for Q geq (n-1)!}
	Supposing that $u(x)$ is a radially symmetric and  normal  solution  to \eqref{normal solution} with $$Q(x)\geq (n-1)!$$ where even integer $n\geq 4$. Then there holds
	$$\int_{\mr^n}e^{nu}\ud x\leq|\mathbb{S}^n|$$
	with equality holds if and only if $Q\equiv(n-1)!$.
\end{theorem}
\begin{proof}
	 Set
	$$v(x):=\frac{2}{(n-1)!|\mathbb{S}^n|}\int_{\mr^n}\log\frac{|y|}{|x-y|}(n-1)!e^{nu(y)}\ud y.$$
	Firstly, we know that $v(x)$ is also radial symmetric following the same argument before. Samely,
	setting $h:=u-v$.
For $n\geq 4$, one has
	\begin{equation}\label{Delat h leq 0}
		\Delta h=\frac{2(n-2)}{(n-1)!|\mathbb{S}^n|}\int_{\mr^n}\frac{1}{|x-y|^2}\left((n-1)!-Q\right)e^{nu}\ud y\leq 0
	\end{equation}
	and there holds
	$$(-\Delta)^{\frac{n}{2}}v=(n-1)!e^{nh}e^{nv}.$$
	For brevity, we denote
	$$\alpha_0:=\frac{2}{(n-1)!|\mathbb{S}^n|}\int_{\mr^n}(n-1)!e^{nh}e^{nv}\ud x=\frac{2}{|\mathbb{S}^n|}\int_{\mr^n}e^{nu}\ud x.$$
	Notice that $(n-1)!e^{nh}e^{nv}=(n-1)!e^{nu}$ is radially symmetric.
	By using Lemma \ref{lem: PH for radial  symmetric}, there exists a sequence $R_i\to\infty$ such that
	\begin{equation}\label{Ph inequa 1}
		\lim_{i\to \infty}\frac{4}{n!|\mathbb{S}^n|}\int_{B_{R_i}(0)}x\cdot \nabla ((n-1)!e^{nh})e^{nv}\ud x= \alpha_0(\alpha_0-2).
	\end{equation}
	Using divergence theorem, there holds
	\begin{align*}
		&\frac{4}{n!|\mathbb{S}^n|}\int_{B_{R_i}(0)}x\cdot \nabla ((n-1)!e^{nh})e^{nv}\ud x\\
		=&\frac{4}{|\mathbb{S}^n|}\int_{B_{R_i}(0)}x\cdot \nabla h e^{nu}\ud x\\
		=&\frac{4}{|\mathbb{S}^n|}\int_0^{R_i}\int_{\mathbb{S}^{n-1}(r)}r\frac{\partial h}{\partial r} e^{nu}\ud\sigma\ud r\\
		=&\frac{4}{|\mathbb{S}^n|}\int_0^{R_i} re^{nu(r)}\int_{\mathbb{S}^{n-1}(r)}\frac{\partial h}{\partial r} \ud\sigma\ud r\\
		=&\frac{4}{|\mathbb{S}^n|}\int_0^{R_i}re^{nu(r)}\int_{B_r(0)}\Delta h\ud x\ud r\\
		\leq &0.
	\end{align*}
	Finally, the estimate \eqref{Ph inequa 1} yields that
	$$\alpha_0\leq 2$$
	i.e.
	$$\int_{\mr^n}e^{nu}\ud x\leq |\mathbb{S}^n|$$
	when the  equality achieves, one has
	$\Delta h\equiv 0$ which yields that
	$Q\equiv (n-1)!$.
	On the other hand, if $Q\equiv (n-1)!$, due to the same reason as before, one has
	$$\int_{\mr^n}e^{nu}\ud x=|\mathbb{S}^n|.$$

\end{proof}

Making  use of  the classification theorem for normal solutions with $Q=(n-1)!$ in \cite{CL91},  \cite{Lin}, \cite{CY MRL}, \cite{WX}, \cite{Mar MZ},  \cite{Xu05} and  Lemma \ref{lema: normal}, we  finish the proof of  Theorem \ref{thm: Q leq n-1 R geq0} and Theorem \ref{thm: Q geq n-1 R geq0}.

\section{Volume comparison for compact manifolds}\label{section: compact}

In this section, we are going to deal with compact manifolds.

{\bf Proof of Theorem \ref{thm: Q_  geq n-1, R_g geq 0 for compact}:}

Firstly, since $R_g\geq 0$, it is easy to see that  the Yambe invariant is non-negative which is defined by
$$Y(M^4,[g]):=\inf_{u\in H^1(M,g)\backslash \{0\}}\frac{\int_{M^4}6|\nabla u|^2+R_gu^2\ud\mu_g}{\left(\int_{M^4}|u|^{4}\ud\mu_g\right)^{\frac{1}{2}}}.$$
It is well-known that the Yamabe invariant has a sharp upper bound
\begin{equation}\label{Y leq S^4}
	Y(M^4,[g])\leq Y(\ms^4, [g_0])=12|\ms^4|^{\frac{1}{2}}
\end{equation}
where $(\ms^4,g_0)$ is the standard sphere.
More details can be found in \cite{Lee-Park}.
Based on the solution of the Yamabe problem, there exists a conformal metric $g_1=e^{2\varphi}g$  called Yambe metric such that the scalar curvature $R_{g_1}$  is a constant satisfying
$$Y(M^4,[g])=\frac{\int_{M^4}R_{g_1}\ud\mu_{g_1}}{V(M^4,g_1)^{\frac{1}{2}}}$$
Up to a scaling, one may assume that
\begin{equation}\label{V_g_1}
	V(M^4,g_1)=|\ms^4|
\end{equation} and then one has
\begin{equation}\label{R_g_1}
	R_{g_1}= Y(M^4,[g])|\ms^4|^{-\frac{1}{2}}.
\end{equation}

By the Chern-Gauss-Bonnet formula, we know that the integral of Q-curvature is conformal invariant i.e.
\begin{equation}\label{Q invariant}
	\int_{M^4}Q_g\ud\mu_{g}=\int_{M^4}Q_{g_1}\ud\mu_{g_1}.
\end{equation}
By using the definition of Q-curvature \eqref{Q_g} and inserting \eqref{R_g_1} as well as \eqref{Y leq S^4}, one has
\begin{align*}
	\int_{M^4}Q_g\ud\mu_{g}=&\int_{M^4}Q_{g_1}\ud\mu_{g_1}\\
	=&-\frac{1}{2}\int_{M^4}|Ric_{g_1}|^2\ud\mu_{g_1}+\frac{1}{6}\int_{M^4}R_{g_1}^2\ud\mu_{g_1}\\
	=&-\frac{1}{2}\int_{M^4}|Ric_{g_1}-\frac{R_{g_1}}{4}g_1|^2\ud\mu_{g_1}+\frac{1}{24}\int_{M^4}R_{g_1}^2\ud\mu_{g_1}\\
	=&-\frac{1}{2}\int_{M^4}|Ric_{g_1}-\frac{R_{g_1}}{4}g_1|^2\ud\mu_{g_1}+\frac{1}{24}Y(M^4,[g])^2\\
	\leq &6|\ms^4|.
\end{align*}
If $Q_g\geq 6$, it is easy to see that
$$V(M,g)\leq|\ms^4|.$$
When the equality holds, we have
$$Q_g\equiv 6, \quad Ric_{g_1}=\frac{R_{g_1}}{4}g_1,\quad Y(M^4,[g])=Y(\ms^4, [g_0]).$$
Then with help of \eqref{Y leq S^4}, \eqref{R_g_1}, we have
$$Ric_{g_1}=3g_1.$$
Then Bishop-Gromov volume comparison theorem shows that
$$V(M^4,g_1)\leq |\ms^4|$$
with equality holds if and only if $(M^4,g_1)$ is isometric to standard sphere $(\ms^4, g_0).$ By using \eqref{V_g_1}, we show that  $(M^4, g_1)$ is isometric  to standard sphere $(\ms^4, g_0).$ In other words, there is a conformal metric $\tilde g=e^{2\psi}g_0$ on $\ms^4$ such that $(M^4, g)$ is isometric to $(\ms^4, \tilde g)$.
Moreover, noting that $Q_g=6$, with help of the classification theorem in \cite{CY MRL}, we show that such conformal metric $\tilde g$
 can be written as  $$\tilde g=\left(\frac{2\lambda}{\lambda^2+|x|^2}\right)^2|dx|^2, \mathrm{on}\; \mr^4$$
for some $\lambda>0$ via a stereographic projection.  Thus we finish our proof.

\section{Applications}\label{section:application}

We are going to show some applications of  these higher order Bol's inequality.
In \cite{PT}, Poliakovsky and  Tarantello consider the equation
\begin{equation}\label{two dim case}
	-\Delta u=(1+|x|^{2p})e^{2u}
\end{equation}
where $p>0$ and $x\in \mathbb{R}^2$
with $\beta=\frac{1}{2\pi}\int_{\mr^2}(1+|x|^{2p})e^{2u}\ud x<+\infty$ to study
selfgravitating strings. They make use of very technical methods to show that
the existence 	of radial solution if and only if
\begin{equation}\label{beta range}
	\max\{2,2p\}<\beta<2+2p.
\end{equation}
Here, we will use Theorem \ref{thm: radial solution for Q geq (n-1)!} to show   for each radial solution, the estimate  \eqref{beta range} holds.
Firstly, due to $(1+|x|^{2p})e^{2u}\in L^1(\mr^2)$, Theorem 2.2 in \cite{Li 23 Q-curvature} shows that the solution $u(x)$ is normal.
Making use of the following theorem, we can easily show that
\eqref{beta range} is necessary for the existence of radial solutions.

For higher order cases, we will consider the normal solutions
\begin{equation}\label{normal for 1+|x|^p}
	u(x)=\frac{2}{(n-1)!|\mathbb{S}^n|}\int_{\mr^n}\log\frac{|y|}{|x-y|}(1+|x|^{np})e^{nu}\ud x+C
\end{equation}
with
$(1+|x|^{np})e^{nu}\in L^1(\mr^n)$.
Set the same notation as before
$$\beta:=\frac{2}{(n-1)!|\mathbb{S}^n|}\int_{\mr^n}(1+|x|^{np})e^{nu}\ud x.$$
\begin{theorem}
	For $p>0$, 	consider the normal solution to \eqref{normal for 1+|x|^p}  with even integer $n\geq2$.
	For each radial solution to \eqref{normal for 1+|x|^p}, there holds
	$$\max\{2,2p\}<\beta<2+2p.$$
\end{theorem}
\begin{proof}
	Making use of Lemma \ref{lem: Pohozaev identity} and Lemma \ref{lem: Qe^nu leq  C|x|^-n},
	the following Pohozaev identity holds
	\begin{equation}\label{PH for 1+|x|^np}
		\beta(\beta-2)=\frac{4p}{(n-1)!|\mathbb{S}^n|}\int_{\mr^n}|x|^{np}e^{nu}\ud x
	\end{equation}
which yields that $\beta>2$.
It is obvious to see that
$$\frac{4p}{(n-1)!|\mathbb{S}^n|}\int_{\mr^n}|x|^{np}e^{nu}\ud x<\frac{4p}{(n-1)!|\mathbb{S}^n|}\int_{\mr^n}(1+|x|^{np})e^{nu}\ud x=2p\beta.$$
Thus the identity \eqref{PH for 1+|x|^np} yields that
\begin{equation}\label{2<beta<2+2p }
	2<\beta <2+2p.
\end{equation}

On the other hand, the identity \eqref{PH for 1+|x|^np} is equivalent to
\begin{equation}\label{voluem e^nu}
	\int_{\mr^n}e^{nu}\ud x=\frac{(n-1)!|\mathbb{S}^n|}{4p}\beta(2+2p-\beta).
\end{equation}
For each radial solution to \eqref{normal for 1+|x|^p}, sicne $1+|x|^{np}\geq 1$, Theorem \ref{thm: radial solution for Q geq (n-1)!} yields that
\begin{equation}\label{voluem upper bound for 1+|x|^np}
	\int_{\mr^n}e^{nu}\ud x<(n-1)!|\mathbb{S}^n|.
\end{equation}
Combining \eqref{voluem e^nu} with \eqref{voluem upper bound for 1+|x|^np}, one has
$$(\beta-2)(\beta-2p)>0.$$
With help of \eqref{2<beta<2+2p }, we finally obtain
$$\max\{2, 2p\}<\beta<2+2p.$$
\end{proof}

With help of the above theorem, we answer an  open problem  in \cite{HM}. Combining with Theorem 1.5 in \cite{HM}, we obtain the following result which generalize the result of Poliakovsky and  Tarantello in \cite{PT} for $n=2$ to higher order cases.
\begin{corollary}
	For $n=4$, the integral equation \eqref{normal for 1+|x|^p} has radial normal solutions if and only if
	$$\max\{2,2p\}<\beta<2+2p.$$
\end{corollary}

\vspace{3em}
{\bf Acknowledgements:}  The  first author would like to thank Professor Xingwang Xu  for his helpful discussion and constant encouragement.     Besides, the first author also wants to thank Professor Yuxin Ge for his useful discussion and hospitality in University Paul Sabatier in Toulouse.

\end{document}